\documentclass[10pt]{elsarticle}
\usepackage{amsmath}
\usepackage{amsthm}
\usepackage{amssymb}
\usepackage{pifont}
\usepackage{amsfonts}
\usepackage{graphicx}
\usepackage{setspace}
\usepackage[left=2cm,top=1.75cm,right=2cm]{geometry}

\newtheorem{lemma}{Lemma}

\newtheorem{theorem}{Theorem}
\numberwithin{equation}{section}
\journal{Springer}

\begin{document}
\begin{frontmatter}
\title{On $q$-analogue of modified Kantorovich-type discrete-Beta operators}
\author[label1,label*]{Preeti Sharma}
\ead{preeti.iitan@gmail.com}
\author[label1,label2]{Vishnu Narayan Mishra}
\ead{vishnunarayanmishra@gmail.com,
vishnu\_narayanmishra@yahoo.co.in}
\address[label1]{Department of Applied Mathematics \& Humanities,
Sardar Vallabhbhai National Institute of Technology, Ichchhanath Mahadev Dumas Road, Surat -395 007 (Gujarat), India}
\address[label2]{L. 1627 Awadh Puri Colony Beniganj, Phase -III, Opposite - Industrial Training Institute, Ayodhya Main Road, Faizabad-224 001, (Uttar Pradesh), India }
\fntext[label*]{Corresponding author}

\begin{abstract}
The present paper deals with the Stancu type generalization of the Kantorovich discrete $q$-Beta operators.  We establish some direct results, which include the asymptotic formula and error estimation in terms of the modulus of continuity and weighted approximation. 
\end{abstract}

\end{frontmatter}

\section{Introduction}
In the last decade, some new generalizations of well known positive linear operators based on $q$-integers were introduced and studied by several authors. Stancu type generalization of positive linear operators studied by several authors \cite{CI,PS1,Rasa2009gonska,Acaramc} and references therein. Our aim is to investigate some approximation properties of a Kantorovich-Stancu type $q$-Beta operators. Discrete Beta operators based on $q$-integers was introduced by Gupta $et~ al.$ in \cite{VG} and they established some approximation results. They also obtained some global direct error estimates for the operators (\ref{eq1}) using the second-order Ditzian-Totik modulus of smoothness and defined and
studied the limit discrete $q$-Beta operator.\\
\indent Gupta $et~ al.$ \cite{VG} introduced discrete $q$-Beta operators as follows:
\begin{equation}\label{eq1}
V_{n,q}(f(t);x)=V_{n}(f(t);q;x)= \frac{1}{[n]_q}\sum_{k=0}^{\infty} p_{n,k}(q;x)f\bigg(\frac{[k]_q}{[n+1]_q q^{k-1}}\bigg),
\end{equation}
where

\begin{equation}\label{eq2}
 p_{n,k}(q;x)= \frac{q^{k(k-1)/2}}{B_q(k+1,n)} \frac{x^k}{(1+x)_q^{n+k+1}}.
\end{equation}
Also, they gave the following equalities:\\
$V_{n}(1;q;x)=1,~~~~~~ V_{n}(t;q;x)=x ~~~~~\text {for every } n\in \mathbb{N}$ and\\
$V_{n}(t^2;q;x)=\bigg(\frac{1}{q[n+1]_q}+1\bigg)x^2+ \frac{x}{[n+1]_q}.$
\newline
In the recent years, applications of $q$-calculus in approximation theory is one of the interesting
areas of research. Several authors have proposed the $q$-analogues of Kantorovich type modification of different linear positive operators and studied their approximation behaviors.\\
In 2013 Mishra $et~al.$ \cite{KJ1} introduced Kantorovich-type modification of discrete $q$-Beta operators
for each positive integer $n$, $q\in (0,1)$ as follows:
\begin{equation}\label{eq3}
V^{*}_{n,q}(f(t);q;x)= \frac{[n+1]_q}{[n]_q}\sum_{k=0}^{\infty} \bigg(\int_{\frac{[k]_q}{[n+1]_q}}^{\frac{[k+1]_q}{[n+1]_q}}f(t)d_qt\bigg) \frac{p_{n,k}(q;x)}{ q^{2k-1}},
\end{equation}
where $f$ is a continuous and non-decreasing function on the interval $[0,\infty), x \in [0,\infty).$\\
\indent It is seen that the operators $V^{*}_{n}$ are linear from the definition of $q$-integral, and since $f$
is a non-decreasing function, $q$-integral is positive, so $V^{*}_{n}$ are positive.\\
\indent The aim of this paper is to present a Kantorovich-Stancu type generalization of the operators given by (\ref{eq3}) and to give some approximation properties.\\
Kantorovich-Stancu type generalization of the operators (\ref{eq3}) is define as follows:

\begin{equation}\label{eq4}
\mathcal{L}^{(\alpha,\beta)}_{n,q}(f(t);q;x)= \frac{[n+\beta+1]_q}{[n+ \beta]_q}\sum_{k=0}^{\infty} \bigg(\int_{\frac{[k]_q}{[n+1]_q}}^{\frac{[k+1]_q}{[n+1]_q}}f\bigg(\frac{[n]_q t+\alpha}{[n]_q+\beta}\bigg)\bigg) d_qt \frac{p_{n,k}(q;x)}{ q^{2k-1}},
\end{equation}
where $p_{n,k}(q;x)$ is define as in (\ref{eq2}).

\section{Preliminaries}
To make the article self-content, here we mention certain basic definitions of $q$-calculus,
details can be found in \cite{ET, KC} and the other recent articles. For each nonnegative
integer $n$, the $q$-integer $[n]_q$ and the $q$-factorial $[n]_q!$ are, respectively, defined by
\begin{equation*} \displaystyle [n]_q = \left\{ \begin{array}{ll} \frac{1-q^n}{1-q}, & \hbox{$q\neq1$}, \\
n,& \hbox{$q=1$}, \end{array} \right.
 \end{equation*}
and

\begin{eqnarray*}
[n]_q!=\left\{
\begin{array}{ll}
[n]_q[n-1]_q[n-2]_q...[1]_q,  & \hbox{$n=1,2,...$},\\
1,& \hbox{$n=0$.}
\end{array}\right. 
\end{eqnarray*}
Then for $q >0$ and integers $n, k, k \geq n \geq 0$, we have\\
$$[n+1]_{q}=1+q[n]_q  ~~~~{ \text {  and }}~~~ [n]_q+q^n[k-n]_q=[k]_q.$$
We observe that

\begin{equation*} (1+x)_q^n=(-x;q)_n= \left\{ \begin{array}{ll} (1+x)(1+qx)(1+q^2x)...(1+q^{n-1}x), & \hbox{$n=1,2,...$},\\
1,& \hbox{$n=0.$} \end{array}\right.\end{equation*} 
Also, for any real number $\alpha$, we have

$$(1+x)_q^\alpha =\frac{(1+x)_q^\infty}{(1+q^\alpha x)_q^\infty}.$$ 
In special case, when $\alpha$ is a whole number, this definition coincides with the above definition.\\
The $q$-Jackson integral and $q$-improper integral defined as

\begin{equation*}\int_0^a f(x) d_qx=(1-q)a\sum_{n=0}^\infty f(aq^n)q^n\end{equation*}
and
\begin{equation*}\int_0^{\infty/A} f(x) d_qx=(1-q)a\sum_{n=0}^\infty f\left(\frac{q^n}{A}\right)\frac{q^n}{A},\end{equation*}
provided sum converges absolutely.\\

\section{Basic results}
\begin{lemma}\label{L1}
\cite{KJ1} The following hold:\\
\begin{itemize}
\item[(i)] $V^{*}_{n,q}(1;q;x)=1,$
\item[(ii)] $V^{*}_{n,q}(t;q;x)=x+\frac{q}{[2]_q[n+1]_q},$
\item[(iii)] $V^{*}_{n,q}(t^2;q;x)=\bigg(\frac{q^{n-2}[n+2]_q}{[n+1]_q}\bigg)x^2+\bigg(\frac{q^{n-1}}{[n+1]_q}+\frac{2q+1}{[n+1]_q [3]_q}\bigg) x+ \frac{q}{[n+1]_q^2[3]_q}.$
\end{itemize}\end{lemma}
Now we give an auxiliary lemma for the Korovkin test functions.

\begin{lemma}\label{L2}
Let $e_m(t) =t^m$, $m = 0,1,2.$ we have
\begin{itemize}
\item[(i)] $\mathcal{L}^{(\alpha,\beta)}_{n}(1;q;x)=1,$
\item[(ii)] $\mathcal{L}^{(\alpha,\beta)}_{n}(t;q;x)=\frac{[n]_q x}{[n]_{q}+\beta} +\frac{q[n]_q+\alpha[2]_q[n+1]_q}{[2]_q([n]_q+\beta)[n+ 1]_q},$
\item[(iii)] $\mathcal{L}^{(\alpha,\beta)}_{n}(t^2;q;x)=\bigg(\frac{[n]_q^2}{([n]_{q}+\beta)^2}\frac{q^{(n-2)}[n+2]_q}{[n+1]_q} \bigg)x^2 + \bigg( \frac{[n]_q^2}{([n]_{q}+\beta)^2}\big\{  \frac{q^{(n-1)}}{[n+1]_q}+\frac{(2q+1)}{[3]_q[n+1]_q}\big\}+\frac{2[n]_q \alpha}{([n]_{q}+\beta)^2} \bigg)x + \bigg( \frac{[n]_q^2}{([n]_{q}+\beta)^2}\frac{q}{[3]_q[n+1]_q^2}+\frac{2q[n]_q\alpha}{[2]_q [n+1]_q ([n]_{q}+\beta)^2}+\frac{\alpha^2}{([n]_{q}+\beta)^2}\bigg).$
\end{itemize}
\end{lemma}

\begin{lemma}
For $f\in C[0,1],$ we have $||\mathcal{L}^{(\alpha,\beta)}_{n} f||\leq ||f||.$
\end{lemma}

\begin{lemma}\label{L3}
From Lemma \ref{L2}, we have\\
$ \mathcal{L}^{(\alpha,\beta)}_{n}((t-x);q;x)=\bigg(\frac{[n]_q }{[n]_{q}+\beta}-1\bigg)x +\frac{q[n]_q+\alpha[2]_q[n+1]_q}{[2]_q([n]_q+\beta)[n+ 1]_q},$\\
$\mathcal{L}^{(\alpha,\beta)}_{n}((t-x)^2;q;x)=\bigg(\frac{[n]_q^2}{([n]_{q}+\beta)^2}\frac{q^{(n-2)}[n+2]_q}{[n+1]_q} +1- \frac{2[n]_q }{[n]_{q}+\beta} \bigg)x^2  + \bigg( \frac{[n]_q^2}{([n]_{q}+\beta)^2}\big\{  \frac{q^{(n-1)}}{[n+1]_q}+\frac{(2q+1)}{[3]_q[n+1]_q}\big\}+\frac{2[n]_q \alpha}{([n]_{q}+\beta)^2}  -\frac{2\big(q[n]_q+\alpha[2]_q[n+1]_q\big)}{[2]_q([n]_q+\beta)[n+ 1]_q} \bigg)x + \bigg( \frac{[n]_q^2}{([n]_{q}+\beta)^2}\frac{q}{[3]_q[n+1]_q^2}+\frac{2 q [n]_q \alpha }{[2]_q [n+1]_q([n]_{q}+\beta)^2}+\frac{\alpha^2}{([n]_{q}+\beta)^2}\bigg).$
\end{lemma}

\begin{lemma}\label{L4}
For $0\leq \alpha\leq \beta$, we have \\
$\mathcal{L}^{(\alpha,\beta)}_{n}((t-x)^2;q;x)\leq \frac{[n+1]_q}{([n]_{q}+\beta)^2}\bigg(\phi^2(x)+\frac{q}{[3]_q[n+1]_q}\bigg),$
where $\phi^2(x)=x(1+x)$.
\end{lemma}

\hspace{-0.6cm} \textbf{Proposition 1.} Let $f$ be a continuous function on $[0,\infty)$ then for $n\rightarrow\infty$, the sequence $\{\mathcal{L}^{(\alpha,\beta)}_{n}(f;q;x)\}$ converges uniformly to $f(x)$ in $[a,b]\subset[0,\infty).$\\
\textbf{Proof.}
For sufficiently large $n$, it is obvious from Lemma \ref{L2} that $\{\mathcal{L}^{(\alpha,\beta)}_{n}(e_0;q;x)\}$, $\{\mathcal{L}^{(\alpha,\beta)}_{n}(e_1;q;x)\}$, $\{\mathcal{L}^{(\alpha,\beta)}_{n}(e_2;q;x)\}$ converges uniformly to 1,  $x$ and $x^2$ respectively on every compact subset of $[0,\infty).$ Thus the required result follows from Bohman-Korovkin theorem.

\section{Some auxiliary results}
Let the space $C_B[0,\infty)$ of all continuous and bounded
functions $f$ on $[0,\infty)$, be endowed with the norm $\|f\|=sup\{ \mid f(x)\mid: x\in[0,\infty)\}.$ Further let us consider the Peetre's K-functional which is defined by
 \begin{equation} K_2(f,\delta)= \inf_{g\in W^2}\{\|f-g\|+\delta \|g''\|\},\end{equation}
where $\delta >0$ and
$W^2_{\infty}=\{g\in C_B[0,\infty):g', g'' \in C_B[0,\infty)\}.$ By the method
as given (\cite{DL} p.177, Theorem 2.4), there exists an absolute constant $C>0$ such that
\begin{equation}\label{1}
K_2(f,\delta) \leq C \omega_2(f,\delta),\end{equation}
 where
 
\begin{equation}\label{2}
\omega_2(f,{\delta})=\sup_{0<h \leq \delta}\sup_{x\in[0,\infty)}\mid f(x+2h)-2f(x+h)+f(x)\mid \end{equation}
is the second order modulus of smoothness of $f\in C_B[0,\infty).$
Also we set

\begin{equation}\label{3}
\omega(f,{\delta})= \sup_{0<h \leq
\delta}\sup_{x\in[0,\infty)}\mid f(x+h)-f(x)\mid.
\end{equation}
We denote the usual modulus of continuity of $f\in C_B[0,\infty)$.

\begin{theorem}\label{the1}
Let $f \in C_B[0,\infty)$, then for all $x\in
[0,\infty)$, there exists an absolute constant $ C >0 $ such that
\begin{equation}|\mathcal{L}^{(\alpha,\beta)}_{n}(f,q;x)-f(x)| \leq C \omega_2\left(f,\sqrt{\delta_n(x) +(\alpha_n(x))^2}\right)+ 
\omega(f,\alpha_n(x)).
\end{equation}
\end{theorem}

\textbf{Proof.}
Let $g\in W_\infty^2$ and $x,t\in [0,\infty).$ By Taylor's expansion, we have
\begin{equation} \label{ac1}
g(t)=g(x)+g'(x)(t-x)+\int_x^t (t-u)g''(u)du.
\end{equation}
Define

\begin{equation}\label{eqn1}
\mathcal{\widetilde{L}^{(\alpha,\beta)}}_{n,q}(f,x) = \mathcal{L}^{(\alpha,\beta)}_{n}(f,q;x) + f(x)-f(\eta(x,q)).
\end{equation}

where $\eta(x,q)= \frac{[n]_q x}{[n]_{q}+\beta} +\frac{q[n]_q+\alpha[2]_q[n+1]_q}{[2]_q([n]_q+\beta)[n+ 1]_q}.$
\newline
Now, we have $\mathcal{\widetilde{L}^{(\alpha,\beta)}}_{n,q}(t-x,x)=0,$ $t\in[0,\infty).$\\
Applying $\mathcal{\widetilde{L}^{(\alpha,\beta)}}_{n,q}$ on both sides of (\ref{ac1}), we get
\begin{eqnarray*}
\mathcal{\widetilde{L}^{(\alpha,\beta)}}_{n,q}(g,x)-g(x) & = & g'(x)\mathcal{\widetilde{L}^{(\alpha,\beta)}}_{n,q}((t-x),x)+\mathcal{\widetilde{L}^{(\alpha,\beta)}}_{n,q} \left(\int_x^t(t-u)g''(u)du,x\right)\\
&=& \mathcal{L}^{(\alpha,\beta)}_{n}\left(\int_x^t(t-u)g''(u) du, x\right)\\&&+\int_x^{\frac{[n]_q x}{[n]_{q}+\beta} +\frac{q[n]_q+\alpha[2]_q[n+1]_q}{[2]_q([n]_q+\beta)[n+ 1]_q}} \bigg(\frac{[n]_q x}{[n]_{q}+\beta} +\frac{q[n]_q+\alpha[2]_q[n+1]_q}{[2]_q([n]_q+\beta)[n+ 1]_q}-u \bigg)g''(u)du.
\end{eqnarray*}

\begin{eqnarray*}
|\mathcal{\widetilde{L}^{(\alpha,\beta)}}_{n,q}(g,x)-g(x)| 
&=& \mathcal{L}^{(\alpha,\beta)}_{n}\left(\int_x^t|(t-u)| |g''(u)| du, x\right)\\&&+\int_x^{\frac{[n]_q x}{[n]_{q}+\beta} +\frac{q[n]_q+\alpha[2]_q[n+1]_q}{[2]_q([n]_q+\beta)[n+ 1]_q}} \bigg|\frac{[n]_q x}{[n]_{q}+\beta} +\frac{q[n]_q+\alpha[2]_q[n+1]_q}{[2]_q([n]_q+\beta)[n+ 1]_q}-u \bigg| |g''(u)|du \\
&\leq & \mathcal{L}^{(\alpha,\beta)}_{n}\left((t-x)^2,x\right)||g''|| +   \left(\frac{[n]_q x}{[n]_{q}+\beta} +\frac{q[n]_q+\alpha[2]_q[n+1]_q}{[2]_q([n]_q+\beta)[n+ 1]_q}-x\right)^2 ||g''||.\\
\end{eqnarray*}

On the other hand from Lemma \ref{L4}, we have
\begin{eqnarray*}
&& \mathcal{L}^{(\alpha,\beta)}_{n}\left((t-x)^2,x\right) +   \left(\frac{[n]_q x}{[n]_{q}+\beta} +\frac{q[n]_q+\alpha[2]_q[n+1]_q}{[2]_q([n]_q+\beta)[n+ 1]_q}-x\right)^2 \\
&\leq & \frac{[n+1]_{q}}{([n]_{q}+\beta)^2}\bigg(\phi^2(x)+\frac{q}{[n+1]_q}\bigg)+   \left(\frac{-\beta x}{[n]_{q}+\beta} +\frac{q[n]_q+\alpha[2]_q[n+1]_q}{[2]_q([n]_q+\beta)[n+ 1]_q}\right)^2\\
&\leq & \frac{4[n+1]_{q}}{([n]_{q}+\beta)^2}\bigg(\phi^2(x)+\frac{q}{[3]_q[n+1]}\bigg).
\end{eqnarray*}

Thus, one can do this
\begin{eqnarray}\label{del1}
\bigg|\mathcal{\widetilde{L}^{(\alpha,\beta)}}_{n,q}(g,x)-g(x)\bigg| 
&\leq & \frac{4[n+1]_{q}}{([n]_{q}+\beta)^2}\bigg(\phi^2(x)+\frac{q}{[3]_q[n+1]}\bigg)||g''||\nonumber\\
&\leq & \frac{4[n+1]_{q}}{([n]_{q}+\beta)^2}\delta_{n}^2(x) ||g''||.
\end{eqnarray}
Where $\delta_{n}^2(x) =\bigg(\phi^2(x)+\frac{q}{[3]_q[n+1]}\bigg)$, we observe that,

\begin{eqnarray*}
\bigg|\mathcal{\widetilde{L}^{(\alpha,\beta)}}_{n,q}(f,x)-f(x)\bigg| &\leq & \bigg|\mathcal{\widetilde{L}^{(\alpha,\beta)}}_{n,q}(f-g,x)-(f-g)(x)\bigg|\\
&+&  \bigg|\mathcal{\widetilde{L}^{(\alpha,\beta)}}_{n,q} (g,x)-g(x)\bigg|+\bigg|f(x)-f\left(\frac{[n]_q x}{[n]_{q}+\beta} +\frac{q[n]_q+\alpha[2]_q[n+1]_q}{[2]_q([n]_q+\beta)[n+ 1]_q}\right)\bigg|\\
 &\leq & \|f-g\|+ \frac{4[n+1]_{q}}{([n]_{q}+\beta)^2}\delta_{n}^2(x) ||g''|| + \omega\bigg(f, \bigg|\frac{-\beta x}{[n]_{q}+\beta} +\frac{q[n]_q+\alpha[2]_q[n+1]_q}{[2]_q([n]_q+\beta)[n+ 1]_q}\bigg|\bigg).
\end{eqnarray*}
Now, taking infimum on the right-hand side over all $g\in W^2$, we obtain

\begin{eqnarray*}
\bigg|\mathcal{L}^{(\alpha,\beta)}_{n}(f,q;x)-f(x)\bigg| &\leq & K_2\left(f, \frac{4[n+1]_{q}}{([n]_{q}+\beta)^2}\delta_{n}^2(x)\right)+ \omega\bigg(f, \bigg|\frac{-\beta x}{[n]_{q}+\beta} +\frac{q[n]_q+\alpha[2]_q[n+1]_q}{[2]_q([n]_q+\beta)[n+ 1]_q}\bigg|\bigg)\\
&\leq& C \omega_2\bigg(f,\frac{2[n+1]_{q}^{(1/2)}}{([n]_q+\beta)}\delta_{n}(x)\bigg) + \omega\bigg(f,\frac{1}{([n]_q+\beta)}\bigg),
\end{eqnarray*}
and so the proof is completed.\\

\section{Weighted approximation}
In this section, we obtain the Korovkin type weighted approximation by the operators defined in (\ref{eq4}). The weighted Korovkin-type theorems were proved by Gadzhiev \cite{AD}. A real function $\rho = 1+x^2$ is called a weight function if it is continuous on $%
\mathbb{R}$ and $\lim\limits_{\mid x\mid \rightarrow \infty }\rho (x)=\infty
,~\rho (x)\geq 1$ for all $x\in \mathbb{R}$.

Let $B_{\rho }(\mathbb{R})$  denote the weighted space of real-valued
functions $f$ defined on $\mathbb{R}$ with the property $\mid f(x)\mid \leq
M_{f}~\rho (x)$ for all $x\in \mathbb{R}$, where $M_{f}$ is a constant
depending on the function $f$. We also consider the weighted subspace $%
C_{\rho }(\mathbb{R})$ of $B_{\rho }(\mathbb{R})$ given by $C_{\rho }(%
\mathbb{R})=\{f\in B_{\rho }(\mathbb{R}){:}$ $f$ is continuous on $\mathbb{R} $\} and $C_{\rho}^{*}[0,\infty)$ denotes the subspace of all functions
 $f\in C_{\rho}[0,\infty)$ for which $\lim\limits_{|x|\rightarrow\infty} \frac{ f(x)}{\rho(x)}$ exists finitely.

\begin{theorem}(See \cite{AD} and \cite{AD1})
\begin{itemize}
\item[(i)] There exists a sequence of linear positive operators $A_n(C_{\rho}\rightarrow B_{\rho})$ such that
\begin{equation}\label{wa1}
\lim_{n\rightarrow \infty}\|A_n(\phi^\nu)- \phi^\nu\|_{\rho}=0,~~\nu=0,1,2
\end{equation}
and a function $f^{*}\in C_{\rho} \backslash C^{*}_{\rho}$  with $\lim\limits_{n\rightarrow \infty} \| A_n(f^{*})- f^{*}\|_{\rho}\geq 1.$
\item[(ii)] If a sequence of linear positive operators $A_n(C_{\rho}\rightarrow B_{\rho})$ satisfies conditions (\ref{wa1}) then
\begin{equation}
\lim_{n\rightarrow \infty}\|A_n(f)- f\|_{\rho}=0, \text{  for every } f\in C^{*}_{\rho}.
\end{equation}
\end{itemize}
\end{theorem}
Throughout this paper we take the growth condition as $\rho(x) = 1 + x^2$ and $\rho_{\gamma}(x) = 1 + x^{2+\gamma},~ x\in [0,\infty), \gamma > 0.$
Now we are ready to prove our next result as follows:

\begin{theorem}
For each $f \in C_{\rho}^{*}[0,\infty)$, we have
\begin{equation*}\lim _{n\rightarrow \infty} \| \mathcal{L}^{(\alpha,\beta)}_{n}(f)-f \|_{\rho} = 0.\end{equation*}
\end{theorem}
\textbf{Proof.} Using the theorem in \cite{AD} we see that it is sufficient to verify the following three conditions
\begin{equation}\label{w1}
\lim_{n\rightarrow \infty} \|
\mathcal{L}^{(\alpha,\beta)}_{n}(t^r;q;x)-x^r\|_{\rho}=0, \text{  } r=0,1,2.
\end{equation}
Since, $\mathcal{L}^{(\alpha,\beta)}_{n}(1;q;x)=1$, the first condition of (\ref{w1}) is satisfied for $r=0$. Now,
\begin{eqnarray*}
\|\mathcal{L}^{(\alpha,\beta)}_{n}(t;q;x)-x\|_{\rho}&=&\sup_{x\in [0,\infty)} \frac{\mid \mathcal{L}^{(\alpha,\beta)}_{n}(t;q;x)-x \mid}{1+x^2}\\
&\leq&  \sup_{x\in [0,\infty)}\frac{x}{1+x^2} \bigg|\bigg(\frac{-\beta
}{[n]_{q}+\beta}\bigg)\bigg| +\frac{q[n]_q+\alpha[2]_q[n+1]_q}{[2]_q([n]_q+\beta)[n+ 1]_q} \\
& \rightarrow& 0 ~~ as ~~[n]_q \rightarrow \infty.
\end{eqnarray*}
Finally,
\begin{eqnarray*}
\|\mathcal{L}^{(\alpha,\beta)}_{n}(t^2;q;x)-x^2\|_{\rho} & = &
 \sup_{x\in [0,\infty)} \frac{x^2}{1+x^2}
\bigg| \bigg(\frac{[n]_q^2}{([n]_{q}+\beta)^2}\frac{q^{(n-2)}[n+2]_q}{[n+1]_q} \bigg) \bigg| \\&&+  \sup_{x\in [0,\infty)} \frac{x}{1+x^2} \bigg| \bigg( \frac{[n]_q^2}{([n]_{q}+\beta)^2}\big\{  \frac{q^{(n-1)}}{[n+2]_q}+\frac{(2q+1)}{[3]_q[n+1]_q}\big\}+\frac{2[n]_q \alpha}{[n]_{q}+\beta} \bigg)\bigg| \\&& + \bigg( \frac{[n]_q^2}{([n]_{q}+\beta)^2}\frac{q}{[3]_q[n+1]_q^2}+\frac{2q[n]_q\alpha}{[2]_q [n+1]_q ([n]_{q}+\beta)^2}+\frac{\alpha^2}{([n]_{q}+\beta)^2}\bigg)\\
&\rightarrow & 0 \text{ as }[n]_q \rightarrow \infty .
\end{eqnarray*}
Thus, from Gadzhievs Theorem in \cite{AD} we obtain the desired result of theorem.
\qed

We give the following theorem to approximate all functions in $C_{x^2}[0,\infty)$.
\begin{theorem}
For each $f\in C_{x^2}[0,\infty)$ and $\alpha
>0$, we have \\
$$\lim\limits_{n \rightarrow \infty} \sup_{x\in[0,\infty)} \frac{\mid\mathcal{L}^{(\alpha,\beta)}_{n}(f;q;x)-f(x)
\mid}{(1+x^2)^{1+\alpha}}=0.$$
\end{theorem}
\textbf{Proof.} For any fixed $x_0>0$,
\begin{eqnarray*}
\sup_{x\in[0,\infty)}\frac{\mid \mathcal{L}^{(\alpha,\beta)}_{n}(f;q;x)-f(x) \mid}{(1+x^2)^{1+\alpha}}&\leq & \sup_{x \leq x_0}\frac{\mid \mathcal{L}^{(\alpha,\beta)}_{n}(f;q;x)-f(x) \mid}{(1+x^2)^{1+\alpha}} + \sup_{x \geq x_0}\frac{\mid \mathcal{L}^{(\alpha,\beta)}_{n}(f;q;x)-f(x) \mid}{(1+x^2)^{1+\alpha}}\\
&\leq&  \|\mathcal{L}^{(\alpha,\beta)}_{n}(f)-f\|_{C[0,x_0]} + \|f\|_{x^2}\sup_{x \geq x_0}\frac{\mid \mathcal{L}^{(\alpha,\beta)}_{n}(1+t^2,x)\mid}{(1+x^2)^{1+\alpha}}\\&&+\sup_{x \geq x_0}\frac{\mid f(x) \mid}{(1+x^2)^{1+\alpha}}.
\end{eqnarray*}
The first term of the above inequality tends to zero from Theorem \ref{t2}. By Lemma \ref{L1}(ii), for any fixed $x_0>0$ it is easily seen that $ \sup_{x\geq x_0} \frac{\mid
\mathcal{L}^{(\alpha,\beta)}_{n}(1+t^2,x)\mid}{(1+x^2)^{1+\alpha}}$
tends to zero as $n \rightarrow \infty$. We can choose $x_0>0$ so
large that the last part of the above inequality can be made small
enough. Thus the proof is completed.
\qed

\section{Error Estimation}
The usual modulus of continuity of $f$ on the closed interval $[0, b]$ is defined
by
$$\omega_b(f,\delta) =\sup_{|t-x|\leq\delta,\, x,t\in[0,b]}|f(t)-f(x)|,\,\,  b>0.$$
It is well known that, for a function $f\in E$, $$ \lim_{\delta\rightarrow 0^+}\omega_b(f,\delta)=0,$$
where\\
$$E:=\left\{f\in C[0,\infty):\lim_{x\rightarrow\infty}\frac{f(x)}{1+x^2}\,\, is\,\, finite \right\}.$$
The next theorem gives the rate of convergence of the operators $\mathcal{L}^{(\alpha,\beta)}_{n}(f,q;x)$ to  $f(x),$ for all $f \in E.$

\begin{theorem}  \label{t2}
Let $f\in E$ and $\omega_{b+1}(f,\delta)$ be its modulus of continuity on the
finite interval $[0,b+1]\subset[0,\infty)$, where $b>0$. Then we have
\begin{equation*}
\| \mathcal{L}^{(\alpha,\beta)}_{n}(f,q;x)-f \|_{C[0,b]} \leq
M_f(1+b^2)\delta_n(b)+2\omega_{b+1}\left(f,\sqrt{\delta_n(b)}\right).
\end{equation*}
\end{theorem}
\textbf{Proof.} 
The proof is based on the following inequality
\begin{equation}\label{t3}
\|\mathcal{L}^{(\alpha,\beta)}_{n}(f,q;x)-f \| \leq
M_f(1+b^2)\mathcal{L}^{(\alpha,\beta)}_{n}((t-x)^2,x)+
\left(1+\frac{\mathcal{L}^{(\alpha,\beta)}_{n}(|t-x|,x)}{\delta}\right)\omega_{b+1}(f,\delta).
\end{equation}
For all $(x,t)\in [0,b]\times[0,\infty):= S.$
To prove (\ref{t3}), we write\\
$$S=S_1\cup S_2:=\{(x,t):0\leq x\leq b,\, 0\leq t \leq b+1\}\cup\{(x,t):0\leq x\leq b,\,  t> b+1\}.$$
If $(x, t)\in S_1,$ we can write
\begin{equation}\label{o}
|f(t)-f(x)|\leq \omega_{b+1}(f,|t-x|)\leq\left(1+\frac{|t-x|}{\delta}\right)\omega_{b+1}(f,\delta)
\end{equation}
where $\delta > 0.$ On the other hand, if $(x, t)\in S_2,$ using the fact that $t-x > 1$,
we have
\begin{eqnarray}\label{o1}
|f(t)-f(x)|&\leq& M_f(1+x^2+t^2)\\
&\leq& M_f(1+3x^2+2(t-x)^2)\nonumber\\
&\leq& N_f(1+b^2)(t-x)^2\nonumber
\end{eqnarray}
where $N_f = 6M_f.$ Combining (\ref{o}) and (\ref{o1}), we get (\ref{t3}).
Now from (\ref{t3}) it follows that
\begin{eqnarray*}
|\mathcal{L}^{(\alpha,\beta)}_{n}(f,q;x)-f(x)|&\leq & N_f(1+b^2)\mathcal{L}^{(\alpha,\beta)}_{n}((t-x)^2,q;x)+\left(1+\frac{\mathcal{L}^{(\alpha,\beta)}_{n}(|t-x|,q;x)}{\delta}\right)\omega_{b+1}(f,\delta)\\
&\leq& N_f(1+b^2)\mathcal{L}^{(\alpha,\beta)}_{n}((t-x)^2,q;x)+\left(1+\frac{{[\mathcal{L}^{(\alpha,\beta)}_{n}((t-x)^2,q;x)]}^{1/2}}{\delta}\right)\omega_{b+1}(f,\delta).\\
\end{eqnarray*}
By Lemma \ref{L4}, we have
$$\mathcal{L}^{(\alpha,\beta)}_{n}(t-x)^2\leq\delta_n(b).$$
\begin{equation*}
\| \mathcal{L}^{(\alpha,\beta)}_{n}(f,q;x)-f \| \leq
N_f(1+b^2)\delta_n(b)+\left(1+\frac{\sqrt{\delta_n(b)}}{\delta}\right)\omega_{b+1}(f,\delta).
\end{equation*}
Choosing $\delta =\sqrt{\delta_n(b)},$ we get the desired estimation.
\qed

Now, we give some estimations of the errors
$|\mathcal{L}_{n}^{(\alpha,\beta)}(f)-f |,$ $n\in \mathbb{N}$ for unbounded functions
by using a weighted modulus of smoothness associated to the space $B_{\rho_{\gamma}}{(\mathbb{R}_+)}$.
The weighed modulus of continuity $\Omega_{\rho_{\gamma}}(f;\delta)$ was defined by L\'{o}pez--Moreno in \cite{LM}.
We consider
\begin{equation}\label{r1}
\Omega_{\rho_{\gamma}}(f;\delta) =\sup_{x\geq 0, 0\leq h\leq \delta} \frac{|f(x+h)-f(x)|}{1+(x+h)^{2+\gamma}},\,\,  \delta>0, \gamma\geq 0.
\end{equation}
It is evident that for each $f\in B_{\rho_{\gamma}}{(\mathbb{R}_+)}, \  \Omega_{\rho_{\gamma}}(f; \cdot)$ is well defined and

$$\Omega_{\rho_{\gamma}}(f;\delta)\leq 2\| f \|_{\rho_{\gamma}}.$$
The weighted modulus of smoothness $\Omega_{\rho_{\gamma}}(f; \cdot)$ possesses
the following properties.\\
$(i)~~~~ \Omega_{\rho_{\gamma}}(f; \lambda \delta)\leq (\lambda+1) \Omega_{\rho_{\gamma}}(f;\delta),\delta>0, \lambda>0$\\
$(ii)~~~~ \Omega_{\rho_{\gamma}}(f; n \delta)\leq n\Omega_{\rho_{\gamma}}(f; \delta), ~~n\in\mathbb{N}$\\
$(iii) ~~~\lim\limits_{\delta \rightarrow  0}\Omega_{\rho_{\gamma}}(f; \delta)= 0.$
\newline
Now, we are ready to prove our next theorem by using above properties.  

\begin{theorem}
For all non-decreasing $f \in B_{\rho_{\gamma}}{(\mathbb{R}_+)}$, we have
\begin{equation*}
| \mathcal{L}_{n}^{(\alpha,\beta)} (f,q;x)-f| \leq \sqrt{\mathcal{L}_{n}^{(\alpha,\beta)}(\nu^2_{x,\gamma};q;x)}\left( 1 + \frac{1}{\delta}\sqrt{\mathcal{L}_{n}^{(\alpha,\beta)}(\Psi_{x}^2;q;x)}\right)\Omega_{\rho_{\gamma}}(f; q;\delta),
\end{equation*} 
$x\geq 0,~ \delta> 0, ~ n\in \mathbb{N},$ where $$ \nu_{x,\gamma}(t):=1+(x+|t-x|)^{2+\gamma}, ~~ \Psi_{x}(t):=|t-x|, ~ t\geq 0.$$
\end{theorem}
\textbf{Proof.}
Let $n\in\mathbb{N}$ and $f \in B_{\rho_\gamma }(\mathbb{R_+}).$ From (\ref{r1}), 
we can write

\begin{eqnarray*}
|f(t)-f(x)| &\leq& \bigg(1+(x + |t-x|)^{2+\gamma}\bigg)\bigg(1+\frac{1}{\delta} |t-x|\bigg)\Omega_{\rho_{\gamma}}(f; \delta)\\
&=&\nu_{x,\gamma}(t)\bigg(1+\frac{1}{\delta}\Psi_x(t)\bigg)\Omega_{\rho_{\gamma}}(f; \delta).
\end{eqnarray*}
Now, applying operator $\mathcal{L}_{n}^{(\alpha,\beta)}$ on above inequality, we get

\begin{eqnarray*}
| \mathcal{L}_{n}^{(\alpha,\beta)} (f,q;x)-f(x)|
&\leq & \Omega_{\rho_{\gamma}}(f; \delta)\mathcal{L}_{n}^{(\alpha,\beta)}\bigg(\nu_{x,\gamma}\bigg(1+\frac{1}{\delta}\Psi_x\bigg);q;x\bigg)
\end{eqnarray*}
\begin{eqnarray} \label{p2}
~~~~~~~~~~~~~~~~~~~~~~~~~\hspace{2cm}&\leq &\Omega_{\rho_{\gamma}}(f; \delta) \bigg(\mathcal{L}_{n}^{(\alpha,\beta)}(\nu_{x,\gamma};q;x)+\mathcal{L}_{n}^{(\alpha,\beta)}\bigg(\frac{\nu_{x,\gamma} \Psi_x}{\delta};q;x\bigg)\bigg).
\end{eqnarray}
By using the Cauchy-Schwartz inequality, we obtain
\begin{eqnarray*}
\mathcal{L}_{n}^{(\alpha,\beta)}\bigg(\frac{\nu_{x,\gamma} \Psi_{x}}{\delta};q;x\bigg) &\leq & \bigg\{\mathcal{L}_{n}^{(\alpha,\beta)}\big(({\nu_{x,\gamma})^2;q;x)}\bigg\}^{1/2} \bigg\{\mathcal{L}_{n}^{(\alpha,\beta)}\bigg(\bigg(\frac{\Psi_x}{\delta}\bigg)^2;q;x\bigg)\bigg\}^{1/2}\\
&=&\frac{1}{\delta}\bigg\{\mathcal{L}_{n}^{(\alpha,\beta)}({\nu^2_{x,\gamma};q;x)}\bigg\}^{1/2} \bigg\{\mathcal{L}_{n}^{(\alpha,\beta)}({\Psi_x}^2;q;x)\bigg\}^{1/2}.
\end{eqnarray*}
Now, by (\ref{p2}), we get
\begin{equation*}
| \mathcal{L}_{n}^{(\alpha,\beta)} (f,q;x)-f| \leq \sqrt{\mathcal{L}_{n}^{(\alpha,\beta)}(\nu^2_{x,\gamma};q;x)}\left( 1 + \frac{1}{\delta}\sqrt{\mathcal{L}_{n}^{(\alpha,\beta)}(\Psi_{x}^2;q;x)}\right)\Omega_{\rho_{\gamma}}(f; \delta).
\end{equation*} 
\qed

\begin{theorem}
Let $0<\alpha\leq 1$ and E be any bounded subset of the interval $[0,\infty)$. If $f\in C_B[0,\infty)\bigcap Lip_L(\alpha)$, then we have
\begin{eqnarray*}
\bigg|\mathcal{L}^{(\alpha,\beta)}_{n}(f,q;x)-f(x)\bigg|&\leq &B \{\delta_n^{\frac{\alpha}{2}}(x)+2(d(x,E))^\alpha\},
\end{eqnarray*}
where $L$ is a constant depending on $\alpha$, $d(x;E)$ is the distance between x and E defined as
$$d(x;E) = inf \{|t-x|;t \in E\,\,and\,\, x \in [0,\infty)\},$$ and $\delta_n(x)$ is as in  (\ref{L3}).
\end{theorem}
\begin{proof}
From the properties of the infimum, there is at least one point $t_o$  in the closure of E, that is
$t_0\in\bar{E},$ such that $$d(x,E)=|t_0-x|.$$
By the triangle inequality we have
\begin{eqnarray*}
\bigg|f(t)-f(x)\bigg|&\leq&\bigg|f(t)-f(t_0)\bigg|+\bigg|f(t_0)-f(x)\bigg|.
\end{eqnarray*}
And \begin{eqnarray*}
\bigg|\mathcal{L}^{(\alpha,\beta)}_{n}(f,q;x)-f(x)\bigg| &\leq &\mathcal{L}^{(\alpha,\beta)}_{n}(|f(t)-f(x)|,x)\\
&\leq&\mathcal{L}^{(\alpha,\beta)}_{n}\left(\bigg|f(t)-f(t_0)\bigg|,x\right)+\mathcal{L}^{(\alpha,\beta)}_{n}\left(\bigg|f(t_0)-f(x)\bigg|,x\right)\\
&\leq &B \left[\mathcal{L}^{(\alpha,\beta)}_{n}\left(|t-t_0|^\alpha,x\right)+|t_0-x|^\alpha\right]\\
&\leq &B \left[\mathcal{L}^{(\alpha,\beta)}_{n}\left(|t-t_0|^\alpha,x\right)+2|t_0-x|^\alpha\right]\\
\end{eqnarray*}
holds. Here we choose $p_1=\frac{2}{\alpha}$ and $p_2=\frac{2}{2-\alpha}$, we get $\frac{1}{p_1}+\frac{1}{p_2}=1$. Then from well-known H\"{o}lder's inequality, we have
\begin{eqnarray*}
\bigg|\mathcal{L}^{(\alpha,\beta)}_{n}(f,q;x)-f(x)\bigg|&\leq & B \left\{\left[\mathcal{L}^{(\alpha,\beta)}_{n}\left(|t-x|^{\alpha p_1},x\right)\right]^{(1/p_1)}
\left[\mathcal{L}^{(\alpha,\beta)}_{n}\left(1^{p_2},x\right)\right]^{(1/p_2)}+2|t_0-x|^\alpha\right\}\\
&=&B\left\{\left[\mathcal{L}^{(\alpha,\beta)}_{n}\left(|t-x|^2,x\right)\right]^{(\alpha/2)}+2|t_0-x|^\alpha\right\}\\
&=&B\{\delta_n^{\alpha/2}(x)+2(d(x,E))^\alpha\}.
\end{eqnarray*}
This completes the proof.
\end{proof}

\section{Global approximation}
For $f\in C[0,1+a],$ the Ditzian--Totik moduli of smoothness of the first and second order are given by
\begin{equation}\label{g1}
\bar{\omega}_{\psi}(f,{\delta})=\sup_{0<h \leq \sqrt{\delta}}\sup_{x+h\psi(x)\in[0,1+a]}\mid f(x+h \psi(x))-f(x)\mid 
\end{equation}
and

\begin{equation}\label{g2}
\omega_2^{\phi}(f,\sqrt{\delta})=\sup_{0<h \leq \sqrt{\delta}}\sup_{x\pm h \phi(x)\in[0,1+a]}\mid f(x+h\phi(x))- 2 f(x)+f(x-h\phi(x)|\mid 
\end{equation}
respectively and the corresponding $K$-functional is defined as

\begin{equation}\label{g3}
\bar{ K}_{2,\phi}(f,\delta) =\inf\{||f-g||+\delta || \psi^2 g''||+\delta^2||g''||:g\in W^2(\phi)\},
\end{equation}
where $\delta >0$ and
$W^2(\phi)=\{g\in C[0,1+a] : g'\in AC[0,1+a], \phi^2 g'' \in C[0,1+a]\}$  and  $g' \in AC_{loc}[0,1+a]$ means that $g$ is differential and $g'$ is absolutely continuous
on every closed interval $[0,1 + a]$. It is well known (\cite{DT} p.24, Theorem 1.3.1) that
\begin{equation}\label{1}
\bar{K}_{2,\phi}(f,\delta) \leq C \omega^{\phi}_2(f,\sqrt{\delta}),
\end{equation}
where $\psi$ is being admissible step-weight function on $[0,1+a]$.

\begin{theorem}
Let $f\in C[0,1+a]$ with $q\in(0,1)$. Then for every $x\in[0,1]$ we have
\begin{eqnarray*}\label{6}
\bigg|\mathcal{L}^{(\alpha,\beta)}_{n}(f,q;x)-f(x)\bigg| &\leq &  C \omega_2^{\phi}\bigg(f, \frac{[n+1]_q^{1/2}}{[n]_q}\bigg)+\bar{\omega}_{\psi}\bigg(f,\frac{1}{(n+\beta)}\bigg).
\end{eqnarray*}
\end{theorem}

\begin{proof} 
Defining the operators $\mathcal{ \widetilde{L}}^{(\alpha,\beta)}$ as in ( \ref{eqn1}) for the function $g \in W^2(\psi)$, we have

\begin{eqnarray}\label{eqn2}
\bigg|\mathcal{\widetilde{L}}^{(\alpha,\beta)}_{n}(g,q;x)-g(x)\bigg| &\leq & \mathcal{L}_{n}^{(\alpha, \beta)}\bigg(\bigg|\int_{x}^{t}(t-v) g''(v)dv\bigg|,q,x\bigg)+     \bigg| \int_{x}^{\frac{[n]_q x}{[n]_{q}+\beta} +\frac{q[n]_q +\alpha[2]_q[n+1]_q}{[2]_q([n]_q+\beta)[n+ 1]_q}}\bigg(\frac{[n]_q x}{[n]_{q}+\beta} \nonumber\\
&&+\frac{q[n]_q+\alpha[2]_q[n+1]_q}{[2]_q([n]_q + \beta)[n+ 1]_q}-v)\bigg)g''(v)dv \bigg|.
\end{eqnarray}
Since the function $\delta_n^2(x)$ is concave on [0,1], for $v=t+\tau (x-t),\tau\in[0,1]$ we obtain
 
$$\frac{|t-v|}{\delta^2{_n}(v)} =\frac{\tau|x-t|}{\delta^2{_n}(t+\tau(x-t))}\leq \frac{\tau|x-t|}{\delta{_n}^2(t)+\tau(\delta{_n}^2(x)-\delta_{n}^2(t))}\leq  \frac{|t-v|}{\delta_{n}^2(x)}.$$
Now using (\ref{eqn2})
 
\begin{eqnarray*}
\bigg|\mathcal{\widetilde{L}}^{(\alpha,\beta)}_{n}(g,q;x)-g(x)\bigg| &\leq & \mathcal{L}_{n}^{(\alpha, \beta)}\bigg(\bigg|\int_{x}^{t}\frac{(t-v)}{\delta_{n}^2(v)}dv\bigg|,q,x\bigg) ||\delta_{n}^2g''|| \\&&+   \bigg| \int_{x}^{\frac{[n]_q x}{[n]_{q}+\beta} + \frac{q[n]_q +\alpha[2]_q[n+1]_q}{[2]_q([n]_q+\beta)[n+ 1]_q}}
\frac{\bigg|\frac{[n]_q x}{[n]_{q}+\beta} +\frac{q[n]_q +\alpha[2]_q[n+1]_q}{[2]_q([n]_q+\beta)[n+ 1]_q}-v\bigg|}{\delta_{n}^2(v)}dv\bigg| ||\delta_{n}^2 g''||\\
&\leq &\frac{1}{\delta_{n}^2(x)}\mathcal{L}_{n}^{(\alpha, \beta)}((t-x)^2;q,x)||\delta^2_{n}g''||+ \frac{1}{\delta_{n}^2(x)} \bigg(\frac{[n]_q x}{[n]_{q}+\beta} +\frac{q[n]_q +\alpha[2]_q[n+1]_q}{[2]_q([n]_q+\beta)[n+ 1]_q}-x\bigg)^2 ||\delta_{n}^2 g''||.
\end{eqnarray*}
from Lemma \ref{L4} and $||\delta_{n}^2 g''(x)|| \leq |\phi^2 g''|+\frac{q}{[n+1]_q}  ||g''(x)||,$ where $x\in[0,1]$, we get

\begin{eqnarray}\label{eqn3}
\bigg|\mathcal{\widetilde{L}}^{(\alpha,\beta)}_{n}(g,q;x)-g(x)\bigg| &\leq & \frac{[n+1]_q}{(n+\beta)^2} \bigg(|\phi^2 g''|+\frac{q}{(n+\beta)^2}  ||g''||\bigg)
\end{eqnarray}
Using (\ref{eqn3}), we have for $f\in C[0,1+a]$

\begin{eqnarray*}
\big|\mathcal{L}^{(\alpha,\beta)}_{n}(f,q;x)-f(x)\big| &\leq & \big|\mathcal{\widetilde{L}}_{n}^{(\alpha, \beta)}(f-g,q,x)|+ |\mathcal{L}_{n}^{(\alpha, \beta)}(g,q,x)-g(x)\big| +\big| g(x)-f(x)|\\&&
 + \bigg|f\bigg(\frac{[n]_q x}{[n]_{q}+\beta} +\frac{q[n]_q +\alpha[2]_q[n+1]_q}{[2]_q([n]_q+\beta)[n+ 1]_q}\bigg)- f(x) \bigg| \\
 &\leq & 4\|f-g\|+ \frac{[n+1]_q}{([n]_q+\beta)^2} ||\phi^2 g''|| +  \frac{q [n+1]_q }{([n]_q+\beta)^4}  ||g''||  +\bigg|f\bigg(\frac{[n]_q x}{[n]_{q}+\beta} +\frac{q[n]_q +\alpha[2]_q[n+1]_q}{[2]_q([n]_q+\beta)[n+ 1]_q}\bigg)- f(x) \bigg| \\
 & \leq & 4\|f-g\|+ \frac{[n+1]_q}{[n]_q^2} ||\phi^2 g''|| +  \frac{[n+1]_q^2}{[n]_q^4}  ||g''|| + \bigg| f\bigg(\frac{[n]_q x}{[n]_{q}+\beta} +\frac{q[n]_q +\alpha[2]_q[n+1]_q}{[2]_q([n]_q+\beta)[n+ 1]_q}\bigg)- f(x) \bigg|\\
&\leq & C\big(\|f-g\|+\delta ||\phi^2 g''||+ \delta^2 || g''||\big) + \bigg|f\bigg(\frac{[n]_q x}{[n]_{q}+\beta} +\frac{q[n]_q +\alpha[2]_q[n+1]_q}{[2]_q([n]_q+\beta)[n+ 1]_q}\bigg)- f(x) \bigg|\\
\end{eqnarray*}
where $\delta= \frac{[n+1]_q}{[n]_q^2}$. Taking the infimum on the right hand side over all $ g \in W^2 (\phi)$, we get

\begin{eqnarray}\label{p7}
\big|\mathcal{L}^{(\alpha,\beta)}_{n}(f,q ; x)-f(x)\big| & \leq &  C K_{2,\phi}\bigg(f,\frac{[n+1]_q}{[n]_q^2} \bigg)+ + \bigg|f\bigg(\frac{[n]_q x}{[n]_{q}+\beta} +\frac{q[n]_q +\alpha[2]_q [n+1]_q}{[2]_q([n]_q+\beta)[n+ 1]_q}\bigg)- f(x) \bigg|.
\end{eqnarray}
Now,
 
\begin{eqnarray*}
\bigg|f\bigg(\frac{q[n]_q +\alpha[2]_q[n+1]_q}{[2]_q([n]_q+\beta)[n+ 1]_q}\bigg)- f(x) \bigg|&=& \bigg|f\bigg(x+\psi(x). \frac{- \beta x [2]_q[n+1]_q + q[n]_q+\alpha[2]_q[n+1]_q}{\psi(x)([2]_q([n]_q+\beta)[n+ 1]_q))}\bigg)- f(x) \bigg|\\
 &\leq & \sup_{t,t+\psi(t)\bigg(\frac{- \beta x [2]_q[n+1]_q + q[n]_q+\alpha[2]_q[n+1]_q}{\psi(x)([2]_q([n]_q+\beta)[n+ 1]_q))}\bigg)\in[0,1+p]}\bigg|f\bigg(t +\\&& \psi(t)\bigg(\frac{- \beta x [2]_q[n+1]_q + q[n]_q+\alpha[2]_q[n+1]_q}{\psi(x)([2]_q([n]_q+\beta)[n+ 1]_q))}\bigg)-f(t)\bigg|\\
& \leq & \omega_\phi \bigg(f,\frac{- \beta x [2]_q[n+1]_q + q[n]_q+\alpha[2]_q[n+1]_q}{\psi(x)([2]_q([n]_q+\beta)[n+ 1]_q))}\bigg)\\
& \leq & \omega_\phi \bigg(f, \frac{1}{(n+\beta)}\bigg).
\end{eqnarray*}
 Hence by (\ref{g1}) and (\ref{p7})
\begin{eqnarray*}
\big|\mathcal{{L}}^{(\alpha,\beta)}_{n}(f,q;x)-f(x)\big| &\leq & C \omega_2^{\phi}\bigg(f, \frac{[n+1]_q^{1/2}}{[n]_q}\bigg)+\bar{\omega}_{\psi}\bigg(f,\frac{1}{(n+\beta)}\bigg).
\end{eqnarray*}
which complete the proof.
\end{proof}

\hspace{-0.6cm}
\textbf{Motivation and applications}\\
In recent years, applications of $q$-calculus in the area of approximation theory and
number theory have been an active area of research. The approximation of functions by linear positive operators is an important research topic in general mathematics and it also provides powerful tools to application areas such as computer-aided geometric design, numerical analysis,
and solutions of differential equations. $q$-Calculus is a generalization of many subjects, such as hypergeometric series, complex analysis and particle physics. Currently it continues being an important subject of study. It has been shown that linear positive operators constructed by $q$-numbers are quite effective as far as the rate of convergence is concerned and we can have some unexpected results, which are not observed for classical case.\\

\hspace{-0.6cm}
\textbf{Conclusion}\\
By using the notion of $q$-integers we introduced Kantorovich-type discrete $q$-Beta operators and investigated some local and global approximation properties of these operators. We obtained the rate of convergence by using the modulus of continuity and also established some direct theorems. These results generalize the approximation results proved for Kantorovich-type discrete $q$-Beta operators which are directly obtained by our results for $q$ = 1.\\
\indent The results of our lemmas and theorems are more general rather than the results of any other
previous proved lemmas and theorems, which will be enrich the literate of applications of
quantum calculus in operator theory and convergence estimates in the theory of approximations
by positive linear operators. The researchers and professionals working or intend to
work in the areas of analysis and its applications will find this research article to be quite
useful. Consequently, the results so established may be found useful in several interesting
situation appearing in the literature on Mathematical Analysis, Applied Mathematics and
Mathematical Physics.

\vspace{0.6cm}
\hspace{-0.6cm}


\textbf{References:}
\begin{thebibliography}{9}
\bibitem{VG} V. Gupta, P.N. Agrawal, D.K. Verma, On discrete $q$-Beta operators. Ann. Univ. Ferrara 57, 39--66 (2011).

\bibitem{KJ1} V.N. Mishra, K. Khatri, L.N. Mishra, Statistical approximation by Kantorovich-type discrete $q$-Beta operators, Advances in Difference Equations, 2013 (1), 1-15.

\bibitem{ET} T. Ernst, The history of $q$-calculus and a new method, U.U.D.M. Department of Mathematics, Uppsala University, Uppsala (2000) Report 2000, 16.

\bibitem{KC} V. Kac, P. Cheung, Quantum Calculus, Universitext, Springer-Verlag, New York (2002).

\bibitem{DT} Z. Ditzian, V. Totik, Moduli of Smoothness, Springer, New York, 1987.

\bibitem{DL} R.A. DeVore and G.G. Lorentz, Constructive Approximation, Springer, Berlin (1993).

\bibitem{AD} A.D. Gadzhiev, Theorems of the type of P. P. Korovkin's theorems (Russian), presented at the international conference on the theory of approximation of functions (Kaluga, 1975). Mat. Zametki 20(5), 781--786 (1976).

\bibitem{AD1} A.D. Gadzhiev, A problem on the convergence of a sequence of positive linear operators on unbounded sets, and theorems that are analogous to P. P. Korovkin's theorem (Russian), Dokl. Akad. Nauk SSSR (218)1001--1004 (1974).

\bibitem{CI} \c{C}. Atakut, \.{I}. B\"{u}y\"{u}kyazici, Stancu type generalization of the Favard-Sz\'{a}sz operators, Applied Mathematics Letters, Volume 23, Issue 12, December 2010, Pages 1479--1482.

\bibitem{PS1} V.N. Mishra, P. Sharma, Approximation by Sz\'{a}sz-Mirakyan-Baskakov-Stancu operators, Afrika Matematika (2014); doi: 10.1007/s13370-014-0288-1.

\bibitem{LM} A.J. L\'{o}pez-Moreno, Weighted simultaneous approximation with Baskakov type operators. Acta 333 Math. Hungar. 104, 143--151 (2004).


\bibitem{Rasa2009gonska} H. Gonska, P. Pitul, I. Ra\c{s}a, General King--Type Operators,  Results in Mathematics (2009).


\bibitem{Acaramc}T. Acar, Asymptotic Formulas for Generalized Sz\'{a}sz--Mirakyan Operators, Applied Mathematics and Computation 263, 233--239.


\end{thebibliography}
\end{document}